\documentclass{amsart}

\newtheorem{definition}{Definition}[section]
\newtheorem{theorem}[definition]{Theorem}
\newtheorem{lemma}[definition]{Lemma}
\newtheorem{proposition}[definition]{Proposition}
\newtheorem{corollary}[definition]{Corollary}
\newtheorem{example}[definition]{Example}
\newtheorem{remark}[definition]{Remark}

\begin{document}

\title{On Existence of Regular Jacobi Structures}
\author{Mahuya Datta and Sauvik Mukherjee}
\address{Stat-Math Unit\\ Indian Statistical Institute\\ Kolkata 700108, India\\ e-mail: mahuya@isical.ac.in;
Stat-Math Unit\\ Indian Statistical Institute\\ Kolkata 700108, India\\ e-mail: mukherjeesauvik@yahoo.com}

\begin{abstract} We prove $h$-principle for locally conformal symplectic foliations and contact foliations on open manifolds. We then interpret the results in terms of regular Jacobi structures on manifolds.\\
\mbox{}
\end{abstract}
\maketitle

\section{Introduction}
It is a classical result due to Gromov which says that an \textit{open} manifold admitting a non-degenerate 2-form must also support a symplectic form. A foliated version of this result was proved by Bertelson in \cite{bertelson}.
\begin{theorem} Any 2-form on a foliated manifold $(M,\mathcal F)$ which is non-degenerate on the leaves is homotopic through the space of such 2-forms to a leafwise symplectic form, provided the foliation $\mathcal F$ satisfies some `openness' condition.\label{bertelson}\end{theorem}
The openness condition on the foliation is a technical one and we refer to \cite{bertelson} for the definition. The theorem above follows as a direct consequence of a general $h$-principle type result proved in \cite{bertelson}. In the same article the author also produced several examples to show that the theorem breaks down without the open-ness condition on the foliation.

Symplectic foliations $(M,\mathcal F)$ are closely related to regular Poisson structures on $M$. Recall that a Poisson structure $\pi$ is a bivector field satisfying the condition $[\pi,\pi]=0$, where the bracket denotes the Schouten bracket of multivector fields \cite{vaisman}. The bivector field $\pi$ induces a vector bundle morphism $\pi^\#:T^*M\to TM$ by $\pi^\#(\alpha)(\beta)=\pi(\alpha,\beta)$ for all $\alpha,\beta\in T^*_xM$, $x\in M$.

The characteristic distribution $\mathcal D=\text{ Image }\pi^\#$, in general, is a singular distribution which, however, integrates to a foliation. The restrictions of the Poisson structure to the leaves of this foliation define symplectic forms on the leaves in a canonical way. Thus, the characteristic foliation is a (singular) symplectic foliation. In fact, the restriction of the Poisson structure to a leaf of the foliation has the maximum rank and so it canonically defines a symplectic form on the leaf. A Poisson bivector field $\pi$ is said to be \emph{regular} if the rank of $\pi^\#$ is constant. In this case the characteristic foliation is a regular symplectic foliation on $M$. On the other hand, given a regular symplectic foliation $\mathcal F$ on $M$ one can associate a Poisson bivector field $\pi$ having $\mathcal F$ as its characteristic foliation. Theorem \ref{bertelson} therefore can be stated as follows:\\

\emph{If $\pi_0$ is a regular bivector field (on a manifold $M$) for which the distribution $\mathcal D=\text{Im\ }{\pi_0}^{\#}$ integrates to a foliation satisfying some open-ness condition. Then $\pi_0$ can be homotoped through regular bivector fields $\pi_t$ to a Poisson bivector field $\pi_1$ having $\mathcal D$ as its characteristic distribution. Moreover, the homotopy of underlying distributions $\text{Im\ }{\pi_t}^{\#}$ remains constant.} \\

In fact, to keep the underlying foliation constant during homotopy, Bertelson had to impose the additional condition on $\mathcal F$. In a recent article, Fernandes and Frejlich \cite{fernandes} have shown that if we allow the underlying foliation to vary then a regular bivector field $\pi_0$ whose characteristic distribution is integrable can be homotoped to a Poisson bivector field, provided the manifold is \emph{open}.
In general, an arbitrary distribution need not be integrable. Taking into account Haefliger's result \cite{haefliger} Fernandes and Frejlich have proved the following.

\begin{theorem}\cite{fernandes} Every regular bivector field $\pi_0$ on an open manifold can be homotoped to a Poisson bivector provided the underlying distribution of $\pi_0$ is homotopic to an integrable one.\label{hprinciple_fernandes}\end{theorem}
The present article is inspired by a comment in \cite{fernandes}, where the authors mention that there should be analogues of Theorem~\ref{hprinciple_fernandes} for foliated locally conformal symplectic manifolds, foliated contact manifolds or more generally for Jacobi manifolds. The main results of this article are stated as follows.

\begin{theorem} Let $M^{2n+q}$ be an open manifold with a codimension $q$ foliation $\mathcal{F}_{0}$ and a 2-form $\omega_0$ which is non-degenerate on the leaves of $\mathcal{F}_{0}$. Fix a de Rham cohomology class $\xi \in H_{deR}^{1}(M)$. Then there exists a homotopy $(\mathcal{F}_t, \omega_t)$ and a closed 1-form $\theta$ representing $\xi$ such that
\begin{enumerate}\item $\omega_t$ is $\mathcal{F}_t$-leafwise non-degenerate
\item $\omega_1$ satisfies the relation $d\omega_1-\theta\wedge \omega_1=0$.
\end{enumerate} In particular, $\omega_1$ is $\mathcal F_1$-leafwise locally conformal symplectic form with the foliated Lee class $\xi_{\mathcal F_1}$ which is the image of $\xi$ under the canonical projection $H_{deR}^1(M)\to H^1(M,\mathcal F)$, where $H^1(M,\mathcal F_1)$ denotes the foliated de Rham cohomology group of $(M,\mathcal F_1)$.\label{lcs}\end{theorem}

As an immediate corollary we can deduce that any open manifold $M$ admits a locally conformal symplectic form $\omega$ with a given Lee class $\xi$ (see Section 2) provided that there exists a non-degenerate 2-form on $M$. Further, when $\xi=0$ we get a globally conformal symplectic form and hence a symplectic form on the manifold, thereby recovering Gromov's theorem.

\begin{theorem} Let $M^{(2n+1)+q}$ be an open manifold and $\mathcal{F}_{0}$ a codimension $q$ foliation on $M$. Let $(\theta_{0},\omega_{0})\in \Omega^{1}(M)\times \Omega^{2}(M)$ be such that its restrictions to the leaves of $\mathcal F_0$ are almost contact structures. Then there exists a homotopy $(\mathcal{F}_{t},\theta_{t},\omega_{t})$ such that $(\theta_{t},\omega_{t})$ is a $\mathcal{F}_{t}$-foliated almost contact structure and $\omega_{1}=d\theta_{1}$. In particular, $\theta_1$ is a leafwise contact form on $(M,\mathcal F_1)$. \label{contact}\end{theorem}

As a corollary we can recover Gromov's result on the existence of a contact form on an open manifold by taking $\mathcal F_0$ to be the trivial foliation \cite{gromov}.

It may be recalled that locally conformal symplectic foliations and contact foliations can be viewed as the characteristic foliations associated with regular Jacobi structures (see Section 2). Thus the last two results throw some light on the existence of regular Jacobi structures on open manifolds.

We organise the paper as follows. In Section 2 we recall basic theory of Jacobi manifolds. We briefly discuss the Holonomic Approximation Theorem in Section 3 as it plays an important role in the proof. The proofs of Theorems \ref{lcs}, \ref{contact} are given in Sections 4 and 5 respectively. In section 6 we interpret Theorems ~\ref{lcs} and ~\ref{contact} in terms of regular Jacobi structure on a manifold.

\section{Preliminaries of Jacobi manifolds}

A Jacobi structure on a smooth manifold $M$ is given by a pairing $(\Lambda,E)$, where $\Lambda$ is a bivector field, that is, a section of $\wedge^{2}TM$ and $E$ is a vector field on $M$, satisfying
\begin{equation}[\Lambda,\Lambda]  =  2E\wedge \Lambda,\ \ \ \ \ \ [E,\Lambda] =  0,\label{jacobi_def}\end{equation}
where [\ ,\ ] is the Schouten bracket. If $E=0$ then $\Lambda$ is a Poisson structure on $M$.

\begin{example}
{\em Let $M$ be a $2n$-dimensional manifold with a symplectic form $\omega$ (i.e. a closed, non-degenerate 2-form). The non-degeneracy condition implies that $b:TM\to T^*M$, given by $b(X)=i_x\omega$ is a vector bundle isomorphism, where $i_X$ denotes the interior multiplication by $X\in TM$. Then $M$ has a Poisson structure defined by
\[\pi(\alpha,\beta)=\omega(b^{-1}(\alpha),b^{-1}(\beta)), \ \text{ for all }\alpha,\beta\in T^*_xM, x\in M.\]}\label{ex_symplectic}\end{example}

\begin{example} {\em A locally conformal symplectic manifold (in short, an l.c.s manifold) is an even dimensional manifold $M$ with a non-degenerate 2-form $\omega$ and a closed 1-form $\theta$ such that
\[ d\omega= \theta \wedge \omega.\]
This equation can be expressed as $d_\theta(\omega)=0$, where $d_\theta=d-\theta\wedge\ $. It can be seen easily that  $d_\theta^2=0$; hence $d_\theta$ is a coboundary operator. The 1-form $\theta$ is called the Lee form of the l.c.s structure $\omega$ and the de Rham cohomology class of $\theta$ is called the Lee class of  $\omega$.
For an l.c.s manifold $(M,\omega,\theta)$ the Jacobi pairing is given by \[\Lambda(\alpha,\beta)=\omega(b^{-1}(\alpha),b^{-1}(\beta))\ \ \text{ and }\ \ \ E=b^{-1}(\theta),\]
where $b:TM\to T^*M$ is defined as in Example~\ref{ex_symplectic}.}\end{example}

\begin{example} {\em Contact manifolds are the odd dimensional counterpart of symplectic manifolds. A 1-form $\eta$ on a $(2n+1)$-dimensional manifold is said to be a contact form on $M$ if $\eta\wedge (d\eta)^n$ is nowhere vanishing.
For contact manifold $(M,\eta)$ the Jacobi pairing is given by \[\Lambda(\alpha,\beta)=d\eta(b^{-1}(\alpha),b^{-1}(\beta)),\ \ \ \mbox{and }\ \ \ \ E=b^{-1}(\eta),\]
where $b:TM\to T^*M$ is the isomorphism given by $b(X)=i_X d\eta+\eta(X)\eta$ for all $X\in TM$.}\end{example}

Let $(M,\Lambda,E)$ be a Jacobi manifold. The bivector field $\Lambda$ defines a bundle homomorphism ${\Lambda}^\#:T^*M\to TM$ by
\[{\Lambda}^\#(\alpha)(\beta)=\Lambda(\alpha,\beta),\]
where $\alpha,\beta\in T^*_xM$, $x\in M$. The Jacobi pair $(\Lambda, E)$ defines a distribution $\mathcal D$ as follows:
\[{\mathcal D}(x)={\Lambda}^\#(T^*_xM)+\langle E_x\rangle, x\in M,\]
where $\langle E_x\rangle$ denotes the subspace of $T_xM$ spanned by the vector $E_x$. In general, $\mathcal D$ is a singular distribution; however, it integrates to a (singular) foliation $\mathcal{F}$ on $M$. This foliation is referred as the characteristic foliation of the Jacobi manifold. If $\mathcal D=TM$ then the manifold is locally conformally symplectic or contact according as the dimension of $M$ is even or odd. More generally, the leaves of the characteristic foliation are either locally conformally symplectic or contact with the induced Jacobi structures on them. In particular, when $E=0$, $\Lambda$ is a Poisson bivector field and the induced Poisson structure on the leaves are non-degenerate. Thus the leaves are symplectic manifolds.

A Jacobi manifold will be called \textit{regular} if its characteristic foliation is regular. Thus a regular Jacobi manifold has a regular l.c.s foliation or a regular contact foliation depending on the dimension of the foliation. Recall that a foliated $k$-form on a foliated manifold $(M,\mathcal F)$ is a section of the exterior bundle $\wedge^k(T^*\mathcal F)$. The exterior differential operator $d$ on the de Rham complex defines in a canonical way a coboundary operator $\bar{d}$ on the foliated de Rham complex $\Gamma (\wedge^kT^*\mathcal F), k\geq 0$. The resulting cohomology is called the foliated de Rham cohomology of the pair $(M,\mathcal F)$. By a foliated l.c.s. structure on a regular foliation $\mathcal F$, we mean a pair $(\omega,\theta)$, where $\theta$ is a foliated closed 1-form and $\omega$ is a non-degenerate foliated 2-form such that $\bar{d}_\theta\omega=0$, where $\bar{d}_\theta=\bar{d}+\theta\wedge \omega$. The cohomology class of $\theta$ in the foliated de Rham cohomology group $H^1(M,\mathcal F)$ is called the Lee class of the foliated l.c.s. structure $\omega$. On the other hand given a regular foliation with leafwise l.c.s or contact form there is a Jacobi structure on $M$ whose characteristic foliation is the given one.

For general theory of symplectic and contact manifolds and Poisson, Jacobi manifolds we refer to \cite{blair}, \cite{lichnerowicz} and \cite{vaisman}.

\section{Holonomic Approximation Theorem}
In this section we review the Holonomic Approximation Theorem from the theory of $h$-principle as it is going to play an important role in the proofs.
Let $p:E\to M$ be a smooth fibration and let $E^{(r)}\to M$ denote the $r$-jet bundle of local sections of $E$.
A section $\sigma$ of $E^{(r)}$ is said to be \textit{holonomic} if it is the $r$-jet of some section $f$ of $E$.

Let $A$ be a polyhedron (possibly non-compact) in $M$ of positive codimension. Let $\sigma$ be any section of the $r$-jet bundle $E^{(r)}$ over $Op\,A$. Here $Op\,A$ denotes an unspecified open neighbourhood of $A$ which may change in course of an argument. The Holonomic Approximation Theorem (see \cite{eliashberg}) says that given any positive functions $\varepsilon$ and $\delta$ on $M$ there exist a (small) diffeotopy $\delta_t$ and a holonomic section $\sigma':Op\,\delta_1(A)\to E^{(r)}$ such that
\begin{enumerate}\item $dist(x,\delta_t(x))<\delta(x)$ for all $x\in M$ and $t\in [0,1]$ and
\item $dist(\sigma(x),\sigma'(x))<\varepsilon(x)$ for all $x\in Op\,(\delta_1(A))$.
\end{enumerate}
This means that under the given conditions any $\sigma$ can be approximated (in the fine $C^0$ topology) by a holonomic section $j^rf$ in a neighbourhood of $\delta_1(A)$, where $\delta_t$ is an arbitrary small diffeotopy of $M$. By a small diffeotopy we mean that $\delta_t(A)$ remains within the domain of $\sigma$ for all $t$

We can summarise the main argument used in the proofs of Theorems~\ref{lcs} and \ref{contact} in the following result.

\begin{theorem} Let $M$ be an open manifold and $p:E\to M$ a smooth fibration. Let $\mathcal R$ be an open subset of the jet space $E^{(r)}$. Then given any section $\sigma$ of $\mathcal R$ there exist a core $K$ of $M$ and a holonomic section $\sigma':Op\,K\to E$ such that the linear homotopy between $\sigma$ and $\sigma'$ lies completely within $\Gamma(\mathcal R)$ over $Op\,K$. \label{h-principle}\end{theorem}

\begin{proof} Since $\mathcal R$ is an open subset of $E^{(r)}$, the space of sections of $\mathcal R$ is an open subset of $\Gamma(E^{(r)})$ in the fine $C^0$-topology. Therefore, given a section $\sigma$ of $\mathcal R$, there exists a positive function $\varepsilon$ satisfying the following condition:
\[\sigma'\in\Gamma(E^{(r)}) \text{ and } \text{dist}\,(\sigma(x),\sigma'(x))<\varepsilon(x)\ \ \ \Rightarrow \sigma' \text{ is a section of } \mathcal R\]
Consider a core $A$ of $M$ and a $\delta$-tubular neighbourhood of $A$ for some positive $\delta$. By the Holonomic Approximation Theorem there exist a diffeotopy $\delta_t$ and a holonomic section $\sigma'$ satisfying $(1)$ and $(2)$. Observe that $K=\delta_1(A)$ is also a core.
Now, suppose that $\sigma$ and $\sigma'$ are $\varepsilon$-close on the $\rho$-neighbourhood $U_\rho$ of $\delta_1(A)$ in the $\delta$-tubular neighbourhood, for some positive number $\rho>0$. Now take a smooth map $\chi :M\rightarrow [0,1]$ satisfying the following conditions:\\
\[\chi  \equiv  1, \text{ on }  U_{\rho/2}\ \ \ \ \text{and}\ \ \ \text{supp\,}\chi \subset U_{\rho}\]
Define a homotopy $\sigma_t$as follows:
\[\sigma_t = \sigma +t\chi (\sigma'-\sigma), \ \ t\in [0,1].\]
Then
\begin{enumerate}
\item $\sigma_0=\sigma$ and each $\sigma_t$ is globally defined;
\item $\sigma_t=\sigma$ outside $U_\rho$ for each $t$;
\item $\sigma_1$ is holonomic on $U_{\rho/2}$. \end{enumerate}
Moreover, since the above homotopy between $\sigma$ and $\sigma'$ is linear, $\sigma_t$ lies in the $\varepsilon$-neighbourhood of $\sigma$ for each $t$. Hence the homotopy $\sigma_t$ lies completely within $\mathcal R$ by the choice of $\varepsilon$. This completes the proof of the proposition.\end{proof}
\begin{remark}{\em In the language of $h$-principle, $\mathcal R$ is called an open differential relation of order $r$. If $\sigma=j^r_f$ is a holonomic section of $\mathcal R$ then $f$ is called a solution of $\mathcal R$. The above theorem is a an instance of local $h$-principle near a core $K$ of $M$. Note that the core $K$ can not be fixed a priori in the statement of the proposition. The theorem is, in fact, true in the general set up where $E$ is a smooth fibration \cite{gromov_pdr}. In this case, however, the linearity condition on the homotopy $\sigma_t$ has to be dropped for obvious reason. }\end{remark}

\section{The foliated l.c.s case}

In this section we prove that an open manifold together with a regular foliation and a leafwise non-degenerate 2-form can be homotoped through such pairs to a regular foliation with a leafwise locally conformal symplectic form having prescribed Lee class.

\begin{lemma} Let $M^{n}$ be a given manifold equipped with a 1-form $\theta$. Then there exists an epimorphism $D_{\theta}:E^{(1)}=(\wedge^{1}M)^{(1)}\rightarrow \wedge^{2}M$ satisfying $D_{\theta}\circ j^{1}\alpha=d_{\theta}\alpha$ so that the following diagram is commutative:
\[\begin{array}{rcl}
  E^{(1)} & \stackrel{D_\theta}{\longrightarrow} & \wedge^{2}(M)\\
  \downarrow &  & \downarrow \\
  M & \stackrel{\text{id}_{M}}{\longrightarrow} & M
 \end{array}\]
In particular, given any 2-form $\omega$ there exists a section $F_\omega:E^{(1)}\to M$ such that $D_\theta\circ F_\omega=\omega$.\label{lemma_lcs}
\end{lemma}
\begin{proof} Let $\theta$ be as in the hypothesis. Define $D_\theta(j^1\alpha(x_0))=d_\theta\alpha(x_0)$ for any local 1-form $\alpha$ on $M$. To prove that the right hand side is independent of the choice of a representative $\alpha$,
choose a local coordinate system $(x^1,...,x^n)$ around $x_{0}\in M$. We may then express $\alpha$ and $\theta$ in the following forms: \[\alpha=\Sigma_{i=1}^n \alpha_idx^i, \ \ \ \theta=\Sigma_{i=1}^n \theta_idx^i\]
where $\alpha_i$ and $\theta_i$ are smooth (local) functions defined in a neighbourhood of $x_0$.
The 1-jet $j^1\alpha(x_0)$ is completely determined by the ordered tuple $(a_i,a_{ij})\in\mathbb{R}^{(n+n^{2})}$, where
\[a_i=\alpha_i(x_0), a_{ij}=\frac{\partial \alpha_{i}}{\partial x^{j}}(x_{0}), \ \ \ i,j=1,2,\dots,n.\]
Now,
\[d_\theta\alpha(x_0)= d\alpha(x_0)-\theta(x_0)\wedge\alpha(x_0)= \Sigma_{i<j}[(a_{ji}-a_{ij})+(\theta_i(x_0)a_j-a_i\theta_j(x_0))]dx^i\wedge dx^j\]
This shows that $d_\theta\alpha(x_0)$ depends only on the 1-jet of $\alpha$ at $x_0$ and the value of $\theta(x_0)$. Since $\theta$ is fixed, $D_\theta$ is well-defined. Hence, $D_{\theta}\circ j^1\alpha = d_\theta(\alpha)_{x_0}$ for any 1-form $\alpha$.

It is easy to check that $D_\theta$ is a vector bundle morphism. In fact, $D_\theta$ is a vector bundle epimorphism. Indeed, given $b_{ij}, 1\leq i<j\leq n$ the following system of linear equations
\[(a_{ij}-a_{ji})+(a_{i}\theta_{j}(x_{0})-a_{j}\theta_{i}(x_{0}))=b_{ij}\]
is clearly solvable by $a_i=0$, $a_{ij}=-a_{ji}=\frac{b_{ij}}{2}$.

Since $D_\theta$ is an epimorphism, every section $\omega :M \rightarrow \wedge^{2}M$ can be lifted up to a section $F_{\omega}:M \rightarrow (T^*M)^{(1)}$ such that $D_{\theta} F_{\omega}=\omega$. Moreover, any two such lifts of a given section $\omega:M\to\wedge^2 T^*M$ are linearly homotopic.\end{proof}

\begin{proposition}Let $M$ be an open manifold and $\mathcal F_0$ be a regular foliation on $M$. Let $\theta$ be a closed 1-form on $M$.
Then any $\mathcal F_0$-leafwise non-degenerate 2-form $\omega_0$ on $M$ can be homotoped through such forms to a 2-form $\omega_1$ which is $d_\theta$-exact on a neighbourhood $U$ of some core $K$ of $M$.\label{approx_lcs}
\end{proposition}
\begin{proof} Let $\mathcal S$ denote the set of all elements $\omega_x$ in $\wedge^2(T_x^*M)$, $x\in M$, such that the restriction of $\omega_x$ is non-degenerate on $D=T_x\mathcal F$. Since non-degeneracy is an open condition, $\mathcal S$ is an open subset of $\wedge^2(T^*M)$.
Given a closed 1-form $\theta \in \xi$, let
\[{\mathcal R}_\theta = D_\theta^{-1}(\mathcal S)\subset E^{(1)},\]
where $D_\theta$ is defined as in Lemma~\ref{lemma_lcs}. Then ${\mathcal R}_\theta$ is an open relation. Let $\sigma_0$ be a section of $E^{(1)}$ such that $D_\theta(\sigma_0)=\omega_0$. By Theorem~\ref{h-principle}, there exists a homotopy of sections $\sigma_t$ of $\sigma_0$ lying in ${\mathcal R}_\theta$ such that $\sigma_1=j^1\alpha$ (that is $\sigma_1$ is holonomic) on an open neighbourhood $U$ of $A$.

Then the 2-forms $\omega_t=D_\theta\circ \sigma_t$, $t\in [0,1]$, take values in $\mathcal S$. Hence,
\begin{enumerate}
\item $\omega_t$ is $\mathcal{F}_{0}$-leafwise non-degenerate for all $t\in [0,1]$.
\item $\omega_1=d_\theta\alpha$ is $d_{\theta}$-exact on $U$.
\end{enumerate}
\end{proof}

\textit{Proof of Theorem~\ref{lcs}.}
To prove the result we proceed as in \cite{fernandes}.
Consider the canonical Grassmann bundle $G_{2n}(TM)\stackrel{\pi}{\longrightarrow}M$ for which the fibres $\pi^{-1}(x)$ over a point $x\in M$ is the Grassmannian of $2n$-planes in $T_x M$. The space $Dist_q(M)$ of codimension $q$ distributions on $M$ can be identified with the section space $\Gamma(G_{2n}(M))$. We topologize $Dist_q(M)$ and $\Gamma(\wedge^2 T^*M)$ with the compact open topology. Let $\Delta_q$ denote the subset of  $Dist_q(M) \times \Gamma(\wedge^2 T^*M)$ defined by
\[\begin{array}{rcl}
\Delta_q &:=&\{(D,\omega):(\iota_D^*\omega)^{n} \neq 0\} \ \ \text{and}\\
\Phi_q&:=& Fol_{q}(M) \times \Gamma(\wedge^2T^*M) \cap \Delta_q
\end{array}\]
where $Fol_q(M) \subset Dist_q(M)$ is the space of all codimension $q$ foliations (that is, integrable distributions) on $M$.

Let $\omega_t$ be a homotopy of 2-forms as obtained in Proposition~\ref{approx_lcs}. Then $(\mathcal{F}_0,\omega_t) \in \Phi_q$ for $0\leq t\leq 1$.
Since $M$ is an open manifold there exists an isotopy $g_t$, $0\leq t\leq 1$, with $g_{0}=id_{M}$ such that $g_1$ takes $M$ into $U$. Now we define $(\mathcal{F}'_t, \omega'_t) \in \Phi_{q}$ for  $t\in [0,1]$ by setting
\begin{center}$\mathcal F'_t=g_t^{-1}\mathcal F_0, \ \ \ \omega'_t=g_t^* \omega_1.$\end{center}
Then, $\omega'_t$ is $\mathcal F'_t$-leafwise non-degenerate. Further, it is easy to see that $\omega'_1$ is $d_{g_1^*\theta}$ closed.
\[\begin{array}{rcl}
d_{g_1^*{\theta}}\omega'_1 & = & d_{g_1^*{\theta}}(g_1^* \omega_1)\\
& = & dg_1^* \omega_1 - g_1^*\theta \wedge g_1^* \omega_1\\
&=& g_1^*[d \omega_1- \theta \wedge \omega_1]\\
&=& g_1^*d_\theta \omega_1=0
\end{array}\]
since $\omega_1$ is $d_\theta$-closed on $U$ and $g_1$ maps $M$ into $U$. Since $g_1$ is homotopic to the identity map the de Rham cohomology class $[g_1^*\theta]=[\theta]=\xi$. Note that $\omega_1$ defines a ${\mathcal F}_1$-leafwise locally conformal symplectic structure.
Let $H^1(\mathcal F_1)$ denote the first foliated de Rham cohomology group of the foliated manifold $(M,\mathcal F)$. Denote by $\xi_{\mathcal F_1}$ the image of $\xi$ under the canonical map $H^1(M)\to H^1(\mathcal F_1)$. Then it follows that the foliated Lee class of the foliated 2-form $\omega'_1$ is $\xi_{\mathcal F_1}$. The desired homotopy is obtained by the concatenation of the two homotopies namely $(\mathcal{F}_0,\omega_t)$ and $(\mathcal{F}'_t,\omega'_t)$.\qed\\

\noindent{\bf Remark} If we take $\theta$ to be the zero form we get back the result of Fernandes and Frejlich in\cite{fernandes}.

\begin{corollary}Let $M$ be a smooth manifold and $\omega$ be a non-degenerate 2-form on $M$. Given any de Rham cohomology class $\xi\in H^1(M)$,
$\omega$ can be homotoped through non-degenerate 2-forms to a locally conformal symplectic form $d_\theta \alpha$, where the deRham cohomology
class of $\theta$ is $\xi$.\label{approx_symp}\end{corollary}

\begin{proof}This is a direct consequence of Theorem~\ref{lcs}.\end{proof}

\section{The foliated contact case}

In this section we prove Theorem~\ref{contact}. The proof is very similar to the l.c.s case.

\begin{definition} {\em Let $M$ be a smooth manifold. A pair $(\alpha,\beta)$  consisting of a 1-form $\alpha$ and a 2-form $\beta$ on $M$ is said to be an almost contact structure on $M$ if $\alpha\wedge \beta^n$ is nowhere vanishing on $M$; equivalently, the restriction of $\beta$ to the vector bundle $\ker \alpha$ is a symplectic structure on the bundle.}
\end{definition}

\begin{lemma} Let $M^n$ be a smooth manifold and $E=\wedge^{1}M$. Then there exists a vector bundle epimorphism $\bar{D}$
\[\begin{array}{rcl}
  E^{(1)} & \stackrel{\bar{D}}{\longrightarrow} & \wedge^{1}M\oplus \wedge^{2}M\\
\downarrow &  & \downarrow \\
M & \stackrel{id_{M}}{\longrightarrow} & M
\end{array}\]
such that $\bar{D}\circ (j^{1}\alpha)=(\alpha,d\alpha)$ for any 1-form $\alpha$ on $M$.\label{lemma_contact}\end{lemma}

\begin{proof} Define $\bar{D}$ by $\bar{D}(j^1\alpha(x_0))=(\alpha(x_0),d\alpha(x_0))$ for any local 1-form $\alpha$ defined near a point $x_0$. It follows from the proof of Lemma~\ref{lemma_lcs} that this map is well defined. Hence $\bar{D}\circ j^1\alpha=d\alpha$ for any 1-form $\alpha$.
Let $(x^{1},...,x^{n})$ be a local coordinate system around $x_{0}\in M$ and $\alpha=\Sigma_{i=1}^{n}\alpha_{i}dx^{i}$ be the representation of $\alpha$ with respect to these coordinates. Then $j^{1}_{x_{0}}(\alpha)$ is uniquely determined by the ordered tuple $(a_{i},a_{ij})\in \mathbb{R}^{n+n^{2}}$ as in Lemma~\ref{lemma_lcs} and
\[\bar{D}(j^{1}_{x_{0}}\alpha)= (\alpha(x_{0}),d\alpha(x_{0}))= (\Sigma_{i=1}^n a_i dx^i, \Sigma_{i<j}(a_{ij}-a_{ji})dx^i \wedge dx^j)\]
Then $\bar{D}$ is  an epimorphism since the following system of equations
\begin{center}
$a_i=b_i$ \ \ and \ \ $a_{ij}-a_{ji}=b_{ij}$ \ \ for all $i\neq j$, $i,j=1,...,n$.
\end{center}
is clearly solvable in $a_i$ and $a_{ij}$.
Hence any section $(\theta,\omega):M\rightarrow \wedge^{1}M\oplus \wedge^{2}M$ can be lifted to a section $F_{(\theta,\omega)}:M\rightarrow E^{(1)}$ such that $\bar{D}\circ F_{(\theta,\omega)}=(\theta,\omega)$ and any two such lifts of a given $(\theta,\omega)$ are homotopic.\end{proof}

\begin{proposition}Let $M$ be an open manifold and $\mathcal F_0$ be a regular foliation on $M$. Let $(\theta,\omega)$ be a pair consisting of a 1-form $\theta$ and a 2-form $\omega$ on $M$ such that the restriction of $(\theta,\omega)$ to the leaves are almost contact structures. Then $(\theta,\omega)$ can be homotoped through such pairs to a pair $(\theta', \omega')$, where $\omega'=d\theta'$ on a neighbourhood $U$ of some core $K$ of $M$.
\end{proposition}
\begin{proof} Let $\mathcal C$ denote the set of all pairs $(\theta_x, \omega_x)\in T_x^*M\times \wedge^2(T_x^*M)$, $x\in M$, such that $\iota_D^*\theta_x\wedge\iota_D^*\omega_x\neq 0$, where $D$ is the tangent space to the leaf of $\mathcal F$ through $x$. Then $\mathcal C$ is an open  subset of $T^*M\oplus \wedge^2(T^*M)$.
Let
\[{\mathcal R} = \bar{D}^{-1}(\mathcal C)\subset E^{(1)},\]
where $E=T^*M$ and $\bar{D}$ is as in Lemma~\ref{lemma_contact}. Then ${\mathcal R}$ is an open relation. Let $\sigma$ be such that $\bar{D}(\sigma_0)=(\theta_0,\omega_0)$. By Theorem~\ref{h-principle}, there exists a homotopy of sections $\sigma_t$ lying in ${\mathcal R}$ such that $\sigma_1=j^1\theta_1$ for some $1$-form $\theta_1$ (that is $\sigma_1$ is holonomic) on an open neighbourhood $U$ of $K$.

Then the pairs $(\theta_t,\omega_t)=D\circ \sigma_t$, $t\in [0,1]$, take values in $\mathcal C$.
It may be seen easily that
\begin{enumerate}
\item $(\theta_t,\omega_t)$ is a $\mathcal{F}_{0}$-leafwise almost contact structure.
\item $d\theta_1=\omega_1$ on $U$.
\end{enumerate}
Hence $(\theta_t,\omega_t)$ is the desired homotopy.
\end{proof}

\textit{Proof of Theorem~\ref{contact}}. Let $D_0$ denote the distribution defined by the foliation $\mathcal F_0$. Define $\Delta_q$, $\Phi_q$ as subsets of $Dist_{q}(M)\times \Omega^1(M) \times \Omega^2(M)$ as follows:
\[\begin{array}{rcl}
\Delta_q &:=&\{(D,\theta,\omega):( \iota _{D}^{*}\theta) \wedge( \iota _{D}^{*}\omega)^{n} \neq 0 \text{ at each point of } M\}\\
\Phi_{q} &:=&Fol_{q}(M)\times \Gamma(\wedge^{1}M\oplus \wedge^{2}M)\cap \Delta_q,\\
 \end{array}\]
where $i_D$ is the inclusion of $D$ in $TM$. By the given hypothesis, $(\mathcal{F}_0,\theta_0,\omega_0)$ is in $\Phi_q$.
Let $(\theta_t,\omega_t)$ be a homotopy of $(\theta_0,\omega_0)$ as obtained in Lemma~\ref{lemma_contact}$. Then (\mathcal{F}_0,\theta_t,\omega_t)$ belongs to $\Phi_{q}$ for $0\leq t\leq 1$.
Choose an isotopy $g_{t}:M\rightarrow M$ such that $g_{0}=id_{M}$ and $g_{1}(M)\subset U$. Such an isotopy exists because $M$ is an open manifold and $U$ is a neighbourhood of a core of $M$. Now we define $(\mathcal{F}'_t,\theta'_t,\omega'_t)\in \Phi_{q}$, $t\in[0,1]$ by setting
\begin{center}$\mathcal{F}'_t= g_t^{-1}(\mathcal{F}_0),\ \ \
\theta'_t=g_t^*\theta_1,\ \ \ \omega'_t=g_t^*\omega_1.$\end{center}
Finally observe that
\[d\theta'_1=d[g_1^*\theta_1]= g_1^*[d\theta_1]= g_1^*[\omega_1]=\omega'_1,\]
since $g_1(M)\subset U$ and $d\theta_1=\omega_1$ on $U$. Therefore, $\theta'_1$ is a ${\mathcal F}'_1$-leafwise contact form. Concatenating the homotopies $(\mathcal F_0,\theta_t,\omega_t)$ and $(\mathcal F'_t,\theta'_t,\omega'_t)$ we obtain the desired homotopy.
\qed\\

\begin{remark} {\em An analogue of Theorem~\ref{bertelson} holds true in the contact case. Suppose that $M$ is a smooth manifold with a foliation $\mathcal F$ of dimension $2n+1$. Let $(\alpha,\beta)$ be a pair consisting of a foliated 1-form $\alpha$ and a foliated 2-form $\beta$ on $(M,\mathcal F)$ such that $\alpha\wedge\beta^n$ is nowhere vanishing on the leaves of $\mathcal F$. The leafwise non-vanishing condition on $(\alpha,\beta)$ is an open condition and hence defines an open subset $\mathcal R$ in the 1-jet space $E^{(1)}$, where $E\to M$ is the vector bundle whose total space is $T^*{\mathcal F}\oplus\Lambda^2(T^*\mathcal F)$. Since the condition is invariant under the action of foliated diffeotopies Bertelson's theorem applies to this relation provided we assume that the foliation $(M,\mathcal F)$ satisfies certain openness condition as defined in \cite{bertelson}. Hence, the given pair $(\alpha,\beta)$ can be homotoped through such pairs to $(\eta,d_{\mathcal F}\eta)$ for some foliated 1-form $\eta$ on $(M,\mathcal F)$, so that $\eta$ is a foliated contact form. (Here $d_\mathcal F$ denotes the coboundary map of the foliated de Rham complex).}\end{remark}

\section{Regular Jacobi structures on open manifolds}

We can reformulate Theorems ~\ref{lcs} and ~\ref{contact} in terms of the Jacobi structure as follows. Let $\nu^k(M)$ denote the space of sections of the alternating bundle $\Lambda^k(TM)$. We shall refer to these sections as $k$-multivector fields on $M$. A pair $(\Lambda,E)\in\nu^2(M)\times\nu^1(M)$ will be called a regular pair if $\mathcal D=\Lambda^\#(T^*M)+\langle E\rangle$ is a regular distribution on $M$.

Let $(\Lambda,E)$ be a regular pair such that the distribution $\mathcal D=\text{Im\,}\Lambda^\#+\langle E\rangle$ is even dimensional. Define $\phi:\mathcal D^*\to \mathcal D$ by $\phi(i^*\alpha)=\Lambda^\#\alpha$ for all $\alpha\in T^*M$.
Using the skew symmetry property of $\Lambda$ we can show that $\ker\Lambda^\#=\ker \alpha$. Hence $\phi$ is well defined. Clearly $\phi$ is an isomorphism and therefore, it defines a section $\omega$ of $\wedge^2\mathcal D^*$ by $\omega(X,Y)=\phi^{-1}(X)(Y)$ which is non-degenerate at each point. Conversely given any section $\omega$ of $\wedge^2\mathcal D^*$, define $\Lambda$ by the following commutative diagram:

\begin{equation}
\begin{array}{rcl}
T^*M & \stackrel{\Lambda^\#}{\longrightarrow} & TM\\
i^* {\downarrow} &  & {\uparrow} i\\
\mathcal D^*& \stackrel{\tilde{\omega}^{-1}}{\longrightarrow} & \mathcal D
\end{array}\label{bivector lambda}
\end{equation}
where $\tilde{\omega}$ is defined by $\tilde{\omega}(X)=i_X\omega$ for all $X\in\Gamma\mathcal D$.

In view of the above correspondence, we can interpret Theorem~\ref{lcs} as follows.
\begin{theorem} Let $M$ be an open manifold and $(\Lambda_0,E_0)$ be a regular pair such that the distribution $\mathcal D=\text{Im\,}\Lambda^\#+\langle E\rangle$ is even dimensional and integrable. Fix a de Rham cohomology class $\xi$ in $H^2(M)$. Then $(\Lambda_0,E_0)$ can be homotoped through regular pairs to a Jacobi pair $(\Lambda_1,E_1)$ such that $\phi_1^{-1} (E_1)$ is a closed 1-form in the foliated de Rham cohomology of $(M,\mathcal F_1)$ and the cohomology class of $\phi_1^{-1} (E_1)$ in $H^1(M,\mathcal F)$ is the same as $\xi_{\mathcal F_1}$, where $\mathcal F_1$ is the characteristic foliation of the Jacobi pair $(\Lambda_1,E_1)$ and $\phi_1:T^*\mathcal F_1\to T\mathcal F_1$ is defined as above.
\label{even_jacobi}\end{theorem}

\begin{theorem} Let $(\Lambda_{0},E_{0})\in \nu^2(M)\times\nu^1(M)$ be a regular pair an open manifold $M$. Suppose that the distribution $\mathcal{D}_0:= \Lambda_0^\#(T^*M)+\langle E_0\rangle$ is odd dimensional. Then $(\Lambda_0,E_0)$ can be homotoped through regular pairs to a Jacobi pair $(\Lambda_1,E_1)$ provided $\mathcal D_0$ is an integrable distribution.\label{odd-jacobi}\end{theorem}
\begin{proof} Let $(\Lambda,E)$ be a regular pair and the distribution $\mathcal D=\Lambda^\#(T^*M)+\langle E\rangle$ is odd dimensional. Define a section $\alpha$ of $\mathcal D^*$ by the relations
\begin{equation}\alpha(\text{Im}\,\Lambda^\#)=0 \ \ \text{ and }\ \ \ \alpha(E)=1.\label{alpha}\end{equation}
Define $\beta$ as a section of $\wedge^2(\mathcal D^*)$ by
\begin{equation}\beta(\Lambda^\#\eta,\Lambda^\#\xi)=\Lambda(\eta,\xi),\ \ \ \ i_E\beta=0\label{beta},\end{equation}
where $i_E$ denotes the interior multiplication by $E$. It can be shown easily that $\beta$ is non-degenerate on Im\,$\Lambda^\#$. Hence  $\alpha\wedge \beta^n$ is nowhere vanishing. Further, note that $\Lambda^\#$ satisfies the relation $i_{\Lambda^{\#}\xi}\beta=-\xi|_{\ker\alpha}$.

On the other hand given a pair $(\alpha,\beta)$ as above we can write $\mathcal D=\ker\alpha\oplus\ker \beta$ and define a vector field $E$ on $M$ satisfying
\begin{equation}i_E\beta=0, \ \ \text{and}\ \ \alpha(E)=1\label{vector E}\end{equation}
Since $\beta$ is non-degenerate on $\ker\alpha$ by our hypothesis, $\tilde{\beta}:\ker\alpha \to (\ker\alpha)^*$ is an isomorphism. For any $\xi\in T^*M$ define $\Lambda^\#(\xi)$ as the unique element of $\ker\alpha$ such that
\begin{equation}\tilde{\beta}(\Lambda^{\#}\xi)=-\xi|_{\ker\alpha}.\label{jacobi_lambda}\end{equation}
Then $(\Lambda,E)$ is a regular pair. Thus there is a one to one correspondence between regular pairs $(\Lambda,E)$ and the triples $(\mathcal D,\alpha,\beta)$ such that $\alpha\wedge\beta^n$ is nowhere vanishing. Further, the regular contact foliations correspond to regular Jacobi pairs with odd-dimensional characteristic distributions under this correspondence \cite{kirillov}.

The result now follows directly from Theorem~\ref{contact}.
Let $(\Lambda_0,E_0)$ be as in the hypothesis and $\mathcal F_0$ be the foliation associated with $\mathcal D_0$. Then we can define $(\alpha_0,\beta_0)$ by the equations (\ref{alpha}) and (\ref{beta}) so that $\alpha_0\wedge \beta_0^n$ is non-vanishing on $\mathcal D_0$. The foliated forms $\alpha_0$ and $\beta_0$ can be extended to differential forms on $M$. Let $\tilde{\alpha}_0$ and $\tilde{\beta}_0$ be any two such extensions of $\alpha_0$ and $\beta_0$ respectively. By Theorem~\ref{contact}, we obtain a homotopy $(\mathcal F_t,\tilde{\alpha}_t,\tilde{\beta}_t)$ of $(\mathcal F_0, \tilde{\alpha}_0,\tilde{\beta}_0)$ such that $\tilde{\alpha}_t\wedge\tilde{\beta}_t^n$ is a nowhere vanishing form on $T\mathcal F_t$ and $\tilde{\beta_1}=d\tilde{\alpha}_1$ on $\mathcal F_1$, so that $\tilde{\alpha}_1$ is a $\mathcal F_1$-leafwise contact form. Let $(\alpha_t,\beta_t)$ be a pair of foliated forms obtained by restricting $(\tilde{\alpha}_t,\tilde{\beta}_t)$ to the tangent distribution of the foliation $\mathcal F_t$. The desired homotopy $(\Lambda_t,E_t)$ is then obtained from  $(\alpha_t,\beta_t)$ by (\ref{vector E}) and (\ref{jacobi_lambda}).
\end{proof}
We conclude with the following remark.
\begin{remark} {\em The integrability condition on the initial distribution in Theorems~\ref{lcs} and \ref{contact} can be relaxed to the extent that it is enough to have the distribution homotopic to the distribution of a foliation. This can be seen by taking into account the classification of foliations due to Haefliger\cite{haefliger}. We refer to  \cite{fernandes} for a detailed argument.}\end{remark}

\end{document}